\documentclass[12pt,letterpaper,reqno]{amsart}
\usepackage{graphicx}
\usepackage{amsmath}
\usepackage{amssymb}
\usepackage{amsfonts}
\usepackage{amsrefs}
\usepackage{enumerate}
\usepackage[mathscr]{eucal}

\usepackage[usenames]{color}

\newcommand{\R}{\mathbb{R}}

\newcommand{\cN}{{\mathcal N}}

\newtheorem{thm}{Theorem}
\newtheorem{lemma}[thm]{Lemma}
\newtheorem{prop}[thm]{Proposition}
\newtheorem{cor}[thm]{Corollary}

\theoremstyle{definition}
\newtheorem{definition}[thm]{Definition}

\theoremstyle{remark}
\newtheorem{remark}[thm]{Remark}
\newtheorem*{ack}{Acknowledgments}

\DeclareMathOperator{\interior}{int}

\DeclareMathOperator{\spt}{spt}
\DeclareMathOperator{\Id}{Id}
\DeclareMathOperator{\vol}{vol}

\begin{document}

\title[Smoothing the Bartnik boundary conditions]{Smoothing the Bartnik boundary conditions and other results on Bartnik's quasi-local mass}
\author{Jeffrey L. Jauregui}

\begin{abstract}
Quite a number of distinct versions of Bartnik's definition of quasi-local mass appear in the literature, and it is not a priori clear that any of them produce the same value in general. In this paper we make progress on reconciling these definitions. The source of discrepancies is two-fold: the choice of  boundary conditions (of which there are three variants) and the non-degeneracy or ``no-horizon'' condition (at least six variants). To address the boundary conditions, we show that given a 3-dimensional region $\Omega$ of nonnegative scalar curvature ($R \geq 0$) extended in a Lipschitz fashion across $\partial \Omega$ to an asymptotically flat 3-manifold with $R \geq 0$ (also holding distributionally along $\partial \Omega$), there exists a smoothing, arbitrarily small in $C^0$ norm, such that $R \geq 0$ and the geometry of $\Omega$ are preserved, and the ADM mass changes only by a small amount. With this we are able to show that the three boundary conditions yield equivalent Bartnik masses for two reasonable non-degeneracy conditions. We also discuss  subtleties pertaining to the various non-degeneracy conditions and produce a nontrivial inequality between a no-horizon version of the Bartnik mass and Bray's replacement of this with the outward-minimizing condition.
\end{abstract}

\maketitle

\section{Introduction}

The primary aim of this paper is to discuss and resolve some of the ambiguities pertaining to the numerous versions of Bartnik's quasi-local mass appearing in the literature. Recall that in general relativity,  the \emph{quasi-local mass problem} is to construct a ``suitable''  or ``sensible'' definition of the mass of a 3-dimensional bounded region in a spacelike hypersurface of a spacetime (cf. \cites{Pe, Ba5}). As is often done, we restrict to spacelike hypersurfaces that are totally geodesic in the spacetime and asymptotically flat (AF), and require the spacetime to satisfy the dominant energy condition. This implies that we are free to consider AF Riemannian manifolds $(M,g)$ of nonnegative scalar curvature without reference to the ambient spacetime. The notion of the \emph{total mass} of $(M,g)$ is well-established to be the ADM mass \cite{adm}; it is the problem of defining the mass of a bounded region, accounting for both the physical matter fields and the gravitational field itself, that remains difficult.

To make the discussion more precise and to recall Bartnik's quasi-local mass formulation, we give a definition.
\begin{definition}
An \emph{allowable region} is a smooth, connected, compact Riemannian 3-manifold $(\Omega,g_-)$ with connected, nonempty boundary $\partial \Omega$, where $g_-$ has nonnegative scalar curvature, and the mean curvature $H_-$ of $\partial \Omega$ in the outward direction is strictly positive.
\end{definition}

In 1989 R. Bartnik proposed a definition of quasi-local mass that essentially localizes the ADM mass \cite{Ba1}. Roughly, the idea is to first consider all AF spaces that ``extend'' $(\Omega, g_-)$ and, second, to minimize the ADM mass within this class. However, a point of confusion is that many distinct definitions of ``extend'' have appeared in the literature. Two aspects of the definition of extension vary: one involves the boundary condition of joining the extension to $\Omega$, which we discuss immediately below. The other is explained later as ``condition $\cN$'' (for ``non-degeneracy'' or ``no horizons'').

\begin{definition}
A smooth AF 3-manifold $(M,g_+)$ is an \emph{extension} of an allowable region $(\Omega, g_-)$ if there exists an isometry $\iota: \partial \Omega \to \partial M$ (with the induced metrics from $g_-$ and $g_+$, respectively). Using $\iota$, we identify $\partial \Omega$ and $\partial M$ (which we will also call $\Sigma$), thus gluing $(\Omega, g_-)$ to $(M,g_+)$ along $\Sigma$ to produce an AF manifold without boundary, denoted $M \cup \Omega$, whose Riemannian metric $G$ arising from $g_+$ and $g_-$ is Lipschitz everywhere and smooth away from $\Sigma$.  Let $H_+$ be the mean curvature of $\Sigma$ with respect to $g_+$. Regarding $H_-$ and $H_+$ as functions on $\Sigma$, 
we say an extension $(M,g_+)$ of $(\Omega, g_-)$ is
\begin{itemize}
\item \emph{Type 1}, if $G$ is smooth across $\Sigma$ (which implies $H_-=H_+$),
\item \emph{Type 2}, if $H_-=H_+$, and
\item \emph{Type 3}, if $H_- \geq H_+$.
\end{itemize}
An extension $(M,g_+)$ is \emph{admissible} if it has nonnegative scalar curvature and is Type 1, 2, or 3. The conditions $H_- = H_+$ and $g_+|_{T \Sigma} = g_-|_{T \Sigma}$ are the \emph{Bartnik boundary conditions}.
\end{definition}

Clearly a Type 1 extension is Type 2, and a Type 2 extension is Type 3. Bartnik originally considered Type 1 extensions, which equivalently may be viewed as smooth AF manifolds without boundary into which $(\Omega,g_-)$ embeds isometrically. For Type 2 and 3 extensions, the metric $G$ may only be Lipschitz across $\Sigma$, but the condition $H_- = H_+$, or even $H_- \geq H_+$, ensures that the scalar curvature is nonnegative across $\Sigma$ in a distributional sense (see, for example, section 2 of \cite{Mi2} for a thorough discussion of this, as well as following Lemma 3.1 of \cite{ST} and section 4 of \cite{Ba4}, for example). Type 2 extensions were first considered by Bartnik as limits of Type 1 extensions \cite{Ba3}. To the author's knowledge, Type 3 extensions were first considered in the positive mass theorem with corners of P. Miao \cite{Mi1} and Y. Shi and L.-F. Tam \cite{ST}, later appearing in the context of the Bartnik mass \cite{Mi2}.

Corresponding to the three types of extensions as above, we consider three versions of Bartnik's definition of quasi-local mass that have appeared in the literature. For $i=1,2,3$, define
\begin{align}
m_{B}^{(i)}(\Omega, g_-) &= \inf \left\{m_{ADM}(M,g_+) \; | \;  (M,g_+) \text{ is a Type $i$ admissible extension of } \right.\nonumber\\
&\qquad \qquad \qquad \qquad \qquad \qquad \left.(\Omega,g_-) \text{ satisfying condition } \cN \right\}.\label{eqn_m_B}
\end{align}
We will discuss condition $\cN$ shortly (and in much greater detail in section \ref{sec_N}). For now, observe that 
$$0 \leq m_B^{(3)}(\Omega, g_-) \leq m_B^{(2)}(\Omega, g_-) \leq m_B^{(1)}(\Omega, g_-),$$
where the first inequality follows from the positive mass theorem with corners \cite[Theorem 1]{Mi1}, \cite[Theorem 3.1]{ST}. (If no Type $i$ admissible extension satisfying condition $\cN$ exists, then the infimum in \eqref{eqn_m_B} is $+\infty$.) 

Why consider Type 2 or Type 3 extensions at all? Bartnik conjectured that the infimum in $m_B^{(1)}$ is achieved, but the minimizer would in general belong to the larger class of Type 2 extensions \cites{Ba1, Ba3} (cf. the discussion in \emph{Prior results} below). Thus, the minimizer would not belong to the class of allowed competitors unless all Type 2 admissible extensions were allowed to begin with. We argue that it is natural to expand to the larger class of Type 3 admissible extensions for the following reason. Any strict positivity of $H_- - H_+$ can be propagated into the interior of the extension as nonnegative scalar curvature; see \cite[Section 3.1]{Mi2}. Conversely, taking a limit of Type 1 admissible extensions can produce a Type 3 admissible extension that is not Type 2 if positive scalar curvature is concentrating near the boundary (as can be seen with simple examples in rotational symmetry). Therefore, allowing Type 2 but disallowing Type 3 extensions seems to be an arbitrary restriction. In any case, the distinction may not matter: it is conjectured (and has been proven for certain choices of $\cN$ --- see below under \emph{Prior results}) that a mass-minimizer among Type 3 admissible extensions is necessarily Type 2, i.e. the Bartnik boundary conditions hold. Also, Theorem \ref{thm_equal_som} establishes one setting in which Types 1, 2, and 3 lead to equivalent definitions of Bartnik mass.  Another reason for considering Type 2 or 3 extensions is the general philosophy that quasi-local mass ought to depend only on the \emph{Bartnik data}, i.e. the induced metric  and mean curvature on $\partial \Omega$, as opposed to the full geometry of $(\Omega, g_-)$. For Type 1 extensions this is not clear (see, however, Corollaries \ref{cor_bartnik_data}, \ref{cor_bartnik_data2}), but for Type 2 and 3 extensions it is (at least for most choices of $\cN$: see section \ref{sec_N}).

\medskip
\paragraph{\emph{Condition $\cN$:}}
Bartnik observed that without some additional hypothesis on admissible extensions, imposed here via condition $\cN$, the value of $m_B^{(1)}(\Omega, g_-)$, when finite, would always be trivially zero \cite{Ba1}.  This is because it is possible to enclose or ``hide''  $(\Omega, g_-)$ behind a ``small neck'' (a compact minimal surface, or \emph{horizon}, of arbitrarily small area) in an admissible extension of arbitrarily small ADM mass. He imposed a ``no horizon'' condition to rule out this type of behavior. This leads to the second point of confusion in the definition of the Bartnik mass: quite a number of distinct ``no horizon'' conditions have appeared in the literature. For instance, are compact minimal surfaces disallowed in $M \cup \Omega$, or just in $M$? Are compact minimal surfaces allowed in $M$ if they do not surround $\Omega$? H. Bray proposed an alternative version of the Bartnik mass, using a different flavor of condition $\cN$ altogether that requires $\partial M$ to be outward-minimizing\footnote{Recall that $\partial M$ is (strictly) outward-minimizing in $(M,g)$ if every surface in $(M,g)$ that encloses $\partial M$ has area at least (strictly greater than) that of $\partial M$.}  in the extension $(M,g_+)$ \cite{Br}.
Unfortunately, there is not even consensus on whether the outward-minimizing condition should be strict or not. These discrepancies could potentially lead to a wide range of distinct Bartnik masses, and reconciling them is not at all straightforward (see sections \ref{sec_N} and \ref{sec_ineq}). According to Bartnik, ``The optimal form of the horizon condition remains conjectural'' \cite{Ba4}. In section \ref{sec_N}, we discuss condition $\cN$ much further (and  will only mention it sparingly until then).

\medskip

The main question that motivates us is whether $m_B^{(1)}$, $m_B^{(2)}$, $m_B^{(3)}$ are all equal on a given allowable region. This obviously depends on the precise choice of condition $\cN$. On a secondary level, we are concerned with how the various choices of $\cN$ affect the value of the Bartnik mass. Our first main result is that Type 3 admissible extensions can be smoothed to Type 1, without altering the ADM mass or the metric very much, or disturbing the geometry of $\Omega$ at all.

\begin{thm}
\label{thm_main}
Suppose $(\Omega, g_-)$ is an extendable allowable region (see Definition \ref{def_extendable}). Let $(M,g_+)$ be some Type 3 admissible extension of $(\Omega, g_-)$. Given any $\epsilon > 0$ and open neighborhood $W$ of $\partial M$ in $M$, there exists a Type 1 admissible extension $(M,\hat g)$ of $(\Omega, g_-)$ such that 
\begin{enumerate}
\item[(i)] the $C^0(M)$ norm of $\hat g - g_+$ is less than $\epsilon$,
\item[(ii)] the $C^2_{-1}(M \setminus W)$ norm of $\hat g - g_+$ is less than $\epsilon$, and
\item[(iii)] the ADM masses of $g_+$ and $\hat g$ differ by less than $\epsilon$.
\end{enumerate}
\end{thm}

$C^2_{-1}$ represents a standard weighted $C^2$ norm, defined in section \ref{sec_background}. Since $\partial \Omega$ generally will have different second fundamental forms with respect to $g_-$ and $g_+$ (even for Type 2 extensions), it is not possible to achieve a $C^1(M)$ smoothing.

\begin{remark}
Theorem \ref{thm_main} leads immediately to a new proof of the positive mass theorem with corners in dimension three \cite[Theorem 1]{Mi1}, \cite[Theorem 3.1]{ST}, (cf. \cite[Theorem 1]{McSz}), for extendable allowable regions with positive boundary mean curvature. Stated in the language of this paper, this says that the ADM mass of a Type 3 admissible extension of an extendable allowable region is nonnegative. This argument of course relies on the smooth version of the positive mass theorem \cites{SY,W}. We do not address the rigidity (zero mass) case here.
\end{remark}

The hope is that Theorem \ref{thm_main} will lead to agreement of $m_{B}^{(1)}, m_{B}^{(2)},$ and $m_{B}^{(3)}$ for a given condition $\cN$. This turns out to be quite delicate in general, and we succeed in this for two versions of condition $\cN$. First:

\begin{thm}
\label{thm_equal_som}
Suppose $(\Omega, g_-)$ is an extendable allowable region, and let $\cN$ be the condition ``$\partial M$ is strictly outward-minimizing.'' Then
$$m_B^{(1)}(\Omega, g_-)= m_B^{(2)}(\Omega, g_-)= m_B^{(3)}(\Omega, g_-).$$
\end{thm}
\noindent Second, in Theorem \ref{thm_equal_eps} we also prove an analogous result for a new modification to the outward-minimizing condition. These two results are proved in section \ref{sec_equiv}.

\medskip
\paragraph{\emph{Prior results:}}
In \cite[Proposition 3.1]{Mi1}, Miao showed that a Type 3 admissible extension may be smoothed locally in a $C^0$ sense. However, this smoothing distorts the metric in $\Omega$ near $\partial \Omega$ and is thus insufficient for the purpose of keeping $(\Omega, g_-)$ fixed. In addition, the smoothing generally introduces negative scalar curvature. It may be removed via a global conformal deformation, but such a deformation distorts the entire geometry of $\Omega$ and perturbs the ADM mass slightly. D. McFeron and G. Sz\'ekelyhidi also obtained a smoothing of a Type 3 admissible extension, using Ricci flow (see the section ``Proof of the main theorem'' in \cite{McSz}). This technique also globally distorts the metric but keeps the ADM mass fixed. (Fixing as opposed to perturbing the ADM mass is useful for proving rigidity statements.)

Miao also proved that the infimum over Type 3 admissible extensions may be restricted, without loss of generality, to the boundary condition $H_- \geq H_+ \geq H_- - \epsilon$ for any $\epsilon>0$, essentially with the choice $\cN =$ ``$\partial M$ is strictly outward-minimizing'' \cite[Proposition 3.3]{Mi2}. He also showed in \cite[Proposition 3.4]{Mi2} that if the infimum in $m_B^{(3)}(\Omega, g_-)$ is achieved, the minimizer is actually a Type 2 admissible extension, hence giving equality of $m_B^{(2)}$ and $m_B^{(3)}$ in this case. His proof assumed that $\partial \Omega$ has positive Gauss curvature. M. Anderson and J. Jauregui proved the latter result without such a hypothesis, using a different approach, for $\cN=$ ``$M$ has no compact minimal surfaces surrounding $\partial M$'' \cite[Theorem 1.1]{AJ}. 

Shortly before this paper was posted to the arXiv, S. McCormick posted \cite{Mc2} to the arXiv, which is on a very similar topic but uses different techniques. In particular, Theorem 3.3 therein corresponds with Theorem \ref{thm_main} above. His Theorem 3.3 requires a convexity condition not present in Theorem \ref{thm_main}, but he obtains a smoothing that only perturbs the metric near $\partial M$ (and thus does not change the ADM mass), provided the metric is not static near $\partial M$. Under the convexity hypothesis, McCormick's Theorem 4.1 gives equality of $m_1, m_2,$ and $m_3$ as in our Theorem \ref{thm_equal_som}; we refer the reader to \cite{Mc2} for the precise details of how the preservation of a non-degeneracy condition (i.e. condition $\cN$) is addressed.

\medskip

We also describe one other result of this paper: In section \ref{sec_ineq}, we show that the value produced by Bray's outward-minimizing version of the Bartnik mass is at least the value produced by the ``no surrounding horizons'' definition, for $m_B^{(3)}$ (Theorem \ref{thm_ineq}). This is apparently a nontrivial fact, as neither of these conditions implies the other, and our proof relies on the Riemannian Penrose inequality \cites{Br,HI} and the recent construction of C. Mantoulidis and R. Schoen \cite{MS}.

\begin{ack}
The author would like to thank M. Anderson, J. Corvino, S. McCormick, and P. Miao for interesting and useful discussions regarding the Bartnik mass. Miao in particular made the author aware of the problem of reconciling the various boundary conditions.
\end{ack}

\section{Preliminaries}
\label{sec_background}

The following definition is slightly non-standard, allowing asymptotically flat metrics to be non-smooth on a compact set.
\begin{definition}
Let $M$ be a smooth, connected, orientable 3-manifold, with compact (possibly empty) boundary. Let $g$ be a continuous Riemannian metric on $M$. Then we say $(M,g)$ is \emph{asymptotically flat} (AF) (with one end) if there exists a compact set $K \subset M$ and a diffeomorphism
$\Phi: M \setminus K \to \R^n \setminus B$, for a closed ball $B$, such that $g|_{M \setminus K}$ is smooth, and in the \emph{asymptotically flat coordinates} $(x^1, x^2, x^3)$ given by $\Phi$, we have
$$g_{ij} = \delta_{ij} + O(|x|^{-p}), \qquad \partial_kg_{ij} =  O(|x|^{-p-1}), \qquad \partial_k\partial_\ell g_{ij} = O(|x|^{-p-2}),$$
for some constant $\frac{1}{2}<p <1$, and the scalar curvature of $g$ in $M \setminus K$ is integrable. (Indices $i,j,k,\ell$ above run from $1$ to $3$, and $\partial$ denotes partial differentiation in the coordinate chart.)
\end{definition}
We regard the order of decay $p$ as above fixed throughout.

\begin{definition}
The \emph{ADM mass} \cite{adm} of an asymptotically flat manifold $(M,g)$ is the real number defined in asymptotically flat coordinates by:
$$m_{ADM}(M,g) = \frac{1}{16\pi} \lim_{r \to \infty} \int_{S_r} \sum_{i,j=1}^3\left( \partial_ig_{ij} - \partial_j g_{ii}\right)\frac{x^j}{r} dA ,$$
where $dA$ is the induced volume form on the coordinate sphere $S_r=\{|x|=r\}$ with respect to $\delta_{ij}$. (It was proved by Bartnik  \cite{Ba0} and Chru\'sciel \cite{Ch} that the ADM mass is well-defined.)
\end{definition}

Let $(M,g)$ be a smooth AF manifold. For an integer $k \geq 0$ and real number $\tau>0$, let $C^k_{-\tau}(M)$ denote the class of $C^k$ functions $f$ on $M$ for which the weighted $C^k$ norm
$$\|f\|_{C^k_{-\tau}(M)} = \sum_{0\leq |\gamma| \leq k} \sup_{x \in M} \sigma(x)^{|\gamma|+\tau} |D^\gamma f(x)|_{g_0}$$
is finite in a fixed AF coordinate chart, where the derivatives are taken with respect to the Levi-Civita  connection of $g_0$. 
Here, $\sigma \geq 1$ is a smooth function on $M$ agreeing with $|x|$ in an asymptotically flat coordinate chart. Thus,  functions in $C^k_{-\tau}(M)$ decay as $O(r^{-\tau})$, with successively faster decay up through the $k$th derivatives.

Abusing notation slightly we also use $\| \cdot \|_{C^k(M)}$ and $\| \cdot \|_{C^k_{-\tau}(M)}$ for tensor norms.

\section{Local extensions of positive scalar curvature.}
\label{sec_local}

\begin{definition}
\label{def_extendable}
An allowable region $(\Omega, g_-)$ is \emph{extendable} if there exists a connected Riemannian 3-manifold $(U, \tilde g)$ of nonnegative scalar curvature (a \emph{local extension}) and an isometric embedding of $(\Omega, g_-)$ onto a compact set in the interior of $(U, \tilde g)$.
\end{definition}

Clearly extendability is necessary for the existence of a Type 1 admissible extension. Note that there exist allowable regions that are not extendable. Counterexamples can be arranged (in rotational symmetry, for instance) for which $(\Omega, g_-)$ has strictly positive scalar curvature $R_-$ in the interior of $\Omega$, but with $R_-$ vanishing on $\partial \Omega$, with $|\nabla R_-| \neq 0$ on $\partial \Omega$. 

The following lemma shows that if $(\Omega, g_-)$ is extendable, then it admits a local extension in which the scalar curvature instantly becomes positive outside $\Omega$.

\begin{lemma}
\label{lemma_local_extension}
An extendable allowable region $(\Omega, g_-)$ admits a local extension $(U, \tilde g)$ such that the scalar curvature of $\tilde g$ is strictly positive on $U \setminus \Omega$ (upon identifying $\Omega$ with its compact embedded image in $U$).
\end{lemma}

\begin{proof}
In some local extension $(U, g)$ of $(\Omega, g_-)$ with $\Omega \subset\subset U$,
consider a tubular neighborhood of $\partial \Omega$ in which the metric $g$ takes the form
\begin{equation}
\label{eqn_g_form}
g = dt^2 + \gamma_t,
\end{equation}
for $t \in (-t_0,t_0)$, for some $t_0 >0$. Here $\gamma_t$ is the induced metric on the surface $\Sigma_t$ of signed distance $t$ from $\partial \Omega$ (where $t < 0$ in the interior of $\Omega$), with respect to $g$. Use local coordinates $(t,x)$ in which $x$ is a coordinate on $\partial \Omega$.

Shrink $U$ if necessary to arrange that 
\begin{equation}
\label{eqn_U}
U \setminus \Omega = \bigcup_{t \in (0, t_0)} \Sigma_t.
\end{equation}
 We will modify $g$ on $U \setminus \Omega$ so that it has positive scalar curvature there. Whenever we shrink $t_0$ below, we implicitly shrink $U$ so that \eqref{eqn_U} is satisfied. 

We will construct a smooth, positive function $\rho$ on $U$ that is identically 1 on $\Omega$, with $\rho=\rho(t)$ on $U \setminus \Omega$, and define
\begin{equation}
\label{eqn_tilde_g}
\tilde g = \rho(t)^2 dt^2 + \gamma_t,
\end{equation}
for $t \in (-t_0,t_0)$. Note $\tilde g$ extends to a smooth Riemannian metric on $U$ that agrees with $g_-$ on $\Omega$.

For the metric $g$ written in the form \eqref{eqn_g_form}, the second-variation-of-area formula (see equation (3) in \cite{Mi1} for example) gives the relation
$$R(t,x) = 2K(t,x) - H(t,x)^2 - \|A(t,x)\|^2-2\frac{\partial H(t,x)}{\partial t},$$
where $K(t,x)$, $H(t,x)$, and $A(t,x)$ are, respectively, the Gauss curvature, mean curvature, and second fundamental form of $\Sigma_t$ with respect to $g$, and $R(t,x)$ is the scalar curvature of $g$. Applying this formula to  the metric $\tilde g$ written in the form \eqref{eqn_tilde_g} leads to the following formula for the scalar curvature $\tilde R$ of $\tilde g$ on $U \setminus \Omega$:
\begin{equation}
\label{R_tilde}
\tilde R(t,x) = \frac{1}{\rho(t)^2} \left(R(t,x) + 2 K(t,x) \left(\rho(t)^2-1\right) + \frac{2\rho'(t)}{\rho(t)} H(t,x)\right).
\end{equation}

To define $\rho$, proceed as follows. Since $H(0,x) = H_-(x) >0$ (by the definition of allowable region), there exists a constant $h_0>0$ such that for $t \in [0,t_0)$ so that $H(t,x)\geq h_0$, shrinking $t_0$ if necessary. Now, define $r(t) = e^{-1/t}$ for $t \in (0, t_0)$; $r(t)$ extends smoothly to $0$ at $t=0$. In particular, $r$ is nonnegative and strictly increasing on $[0,t_0)$.
Finally, let $\kappa_0> 0$ be a constant so that $|K(t,x)| \leq \kappa_0$ (shrinking $t_0$ if necessary) for all $t \in [0,t_0)$ and $x$.
Define, for $0 \leq t < t_0$,
$$\rho(t) = \exp\left(\frac{1}{h_0} \int_0^t r(\tau) d\tau \right).$$

Clearly $\rho(0)=1$ and its derivatives vanish to infinite order at $t=0$, since $r$ and all of its  derivatives  vanish at $t=0$. In particular, $\rho$ extends smoothly by 1 to all of $\Omega$. Also, note that $\rho'(t)>0$ for $t>0$ and consequently $\rho\geq 1$ on $U$. Thus, $\tilde g$ defined in \eqref{eqn_tilde_g}, with this choice of $\rho$, is a smooth Riemannian metric on $U$. We will use \eqref{R_tilde} to estimate its scalar curvature. First we estimate
\begin{align*}
\rho(t)^2-1 &= \exp\left(\frac{2}{h_0} \int_0^t r(\tau) d\tau \right)-1\\
&\leq \exp\left(\frac{2}{h_0} tr(t) \right)-1\\
&\leq \frac{4}{h_0} tr(t)
\end{align*}
for $t \in [0, t_0)$, shrinking $t_0>0$ if necessary. 
Then for $t \in [0,t_0)$,
\begin{align*}
\left|2 K(t,x) \left(\rho(t)^2-1\right)\right| \leq  \frac{8\kappa_0}{h_0} tr(t).
\end{align*}
Now using \eqref{R_tilde} and the fact that $R\geq 0$, we have that for $t \in [0,t_0)$,
\begin{align*}
\tilde R(t,x) &\geq \frac{1}{\rho(t)^2} \left( - \frac{8\kappa_0}{h_0} tr(t) +\frac{2r(t)}{h_0} H(t,x)\right)\\
&\geq \frac{r(t)}{\rho(t)^2} \left(-\frac{8\kappa_0}{h_0} t  +2\right),
\end{align*}
which is strictly positive for $t \in (0,t_0)$, again shrinking $t_0>0$ if necessary. Thus, $(U, \tilde g)$ is a local extension of $(\Omega, g_-)$ for which the scalar curvature in $U \setminus \Omega$ is strictly positive.
\end{proof}

\section{Proof of Theorem \ref{thm_main}}

We outline the proof of Theorem \ref{thm_main} as follows. First, via a perturbation in the weighted space $C^k_{-1}(M)$, we show that without loss of generality, a Type 3 admissible extension may be assumed to obey the strict inequality $H_- > H_+$, and in addition have positive scalar curvature near $\partial M$ (Lemma \ref{lemma_wlg}). Second, we use the local  extension $\tilde g$ in which the scalar curvature instantly becomes positive outside $\Omega$ constructed in Lemma \ref{lemma_local_extension}. The strategy then is to consider a surface $\Sigma_t$ pushed out into $M$ slightly from $\partial \Omega$, which now has a neighborhood on which the scalar curvatures in both the extension $(M,g_+)$ and the local extension are bounded below by a positive constant. A small local perturbation of $g_+$ changes this metric to agree with $\tilde g$ on $\Sigma_t$, while still respecting the strict mean curvature inequality. Finally, a scalar curvature deformation result of S. Brendle, F. Marques, and A. Neves \cite{BMN} is invoked to smoothly match the metrics.

\subsection{A simplification}
First we show that a Type 3 admissible extension can be perturbed slightly so that it has positive scalar curvature near the boundary and satisfies the mean curvature inequality strictly. The technique to create a jump in mean curvature was used by Miao in \cite{Mi2}.

\begin{lemma}
\label{lemma_wlg}
Let $(M,g)$ be a smooth AF 3-manifold with nonnegative scalar curvature and with nonempty compact boundary $\partial M$. Given any $\epsilon > 0$ and integer $k \geq 0$, there exists a smooth asymptotically flat metric $g'$ on $M$ such that
\begin{enumerate}
\item[(i)] $\|g'-g\|_{C^k_{-1}(M)} < \epsilon$, 
\item[(ii)] $g'$ has nonnegative scalar curvature that is strictly positive on a neighborhood of $\partial M$,
\item[(iii)] $g=g'$ on $\partial M$,
\item[(iv)] the mean curvatures of $\partial M$ with respect to $g$ and $g'$, say $H$ and $H'$ respectively, satisfy $H' < H$ pointwise, and
\item[(v)] the ADM masses of $g$ and $g'$ differ by less than $\epsilon$.
\end{enumerate}
\end{lemma}

\noindent Thus, if $(M,g)$ is a Type 3 admissible extension of some allowable region $(\Omega, g_-)$, then so is $(M, g')$.

\begin{proof}
Let $(M,g)$ be given as above. Let $\varphi$ be the unique solution to
$$\begin{cases}
\Delta \varphi = 0 & \text{ on } (M,g)\\
\varphi \to 0 & \text{ at infinity }\\
\varphi =1 & \text{ on } \partial M,
\end{cases}$$
which is positive by the maximum principle. (The existence of $\varphi$ is standard, and can be derived from, for example, the same techniques as in the proof of \cite[Lemma 3.2]{SY}.) Moreover, by standard weighted elliptic estimates (see \cite[Theorem 1]{Meyers}, for example), we have $\varphi$ is smooth and $\varphi \in C^k_{-1}(M)$. For a parameter $a \in (0,1)$ to be determined, let
$$u_a = (1-a)\varphi + a,$$
which is clearly positive and solves
$$\begin{cases}
\Delta u_a = 0 & \text{ on } (M,g)\\
u_a \to a & \text{ at infinity}\\
u =1 & \text{ on } \partial M.
\end{cases}$$
Note $u_a$ converges to 1 smoothly on compact sets as $a \nearrow 1$.  Let $\{\Phi_a\}_{\{\frac{1}{2} \leq a\leq 1\}}$ be a smooth family of diffeomorphisms $M \to M$, equal to the identity near $\partial M$, with $\Phi_a(x) = a^{-2}x$ in some AF coordinate chart, and $\Phi_1 = \Id_M$. 

Consider the metric $g'_a =  \Phi_a^*(u_a^4 g)$. (We use the diffeomorphisms here 
to account for the conformal factors not approaching 1 at infinity.) Note, $g_a'$ is AF with 
nonnegative scalar curvature, since $u_a$ is harmonic with respect 
to $g$. It is straightforward to check that $g'_a$ converges to $g$ 
in $C^k_{-1}(M)$ as $a \nearrow 1$. Also, $g_a'=g$ on $\partial M$ and induces boundary mean curvature 
\begin{equation}
\label{eqn_mc_conf}
H_a' = H + 4\nu(u_a),
\end{equation}
where $\nu$ is the $g$-unit normal to $\partial M$ pointing into $M$. By the Hopf maximum principle, $\nu(u_a)<0$, so that $H_a' < H$ on $\partial M$.  Also, the ADM mass of $g'_a$ converges to that of $g$ as $a \nearrow 1$. This can be seen from a well-known formula that follows directly from the definition of the ADM mass:
\begin{align}
m_{ADM}(g_a') &= m_{ADM}(u_a^4g) = a^2 \left(m_{ADM}(g) - \frac{1}{2\pi a} \lim_{r \to \infty} \int_{S_r} \partial_{\nu} u_a dA\right) \nonumber\\
&=a^2 \left( m_{ADM}(g) - \frac{1-a}{2\pi a} \lim_{r \to \infty} \int_{S_r} \partial_{\nu} \varphi dA\right). \label{eqn_adm_change}
\end{align}
Thus for $a$ sufficiently close to 1, $g'_a$ satisfies properties (i), (iii), (iv), (v), and has nonnegative scalar curvature. Fix such a value of $a$, and let $g'=g'_a$, with scalar curvature $R'\geq 0$.

To achieve (ii), we perform another conformal deformation. Let $\psi \geq 0$ be a smooth, compactly supported  function on $M$ that is strictly positive on a neighborhood of $\partial M$. For a parameter $b>0$ to be determined, consider Poisson's equation
$$\begin{cases}
\Delta w = -b \psi & \text{ on } (M,g')\\
w \to 1 & \text{ at infinity}\\
w =1 & \text{ on } \partial M.
\end{cases}$$
Again by standard arguments in elliptic theory, there exists a unique solution $w=w_b$, which is smooth and positive, and such that $w_b$ converges to 1 in $C^{k}_{-1}(M)$ as $b \searrow 0$.  Consider the conformal metric $g_b'' = w_b^4 g'$. For $b>0$ sufficiently small, this metric satisfies (i) and (iii)--(v). To see (ii), its scalar curvature is given by
$$R_b'' = w_b^{-5} (-8\Delta w_b + R'w_b),$$
which is nonnegative on $M$ and positive near $\partial M$.
\end{proof}

\subsection{The proof of Theorem \ref{thm_main}}
Let $(M,g_+)$ be a Type 3 admissible extension of an extendable allowable region $(\Omega,g_-)$. Without loss of generality, by Lemma \ref{lemma_wlg}, we may assume $H_+ < H_-$ on $\Sigma := \partial M \cong \partial \Omega$ and that the scalar curvature $R_+$ of $g_+$ is strictly positive on a neighborhood of $\partial M$ in $M$. 

Take a local extension $(U, \tilde g)$ of $(\Omega, g_-)$ as in Lemma \ref{lemma_local_extension}. View $U$ as a precompact subset of $M \cup \Omega$ that contains $\Omega$, and extend $\tilde g$ arbitrarily to a smooth Riemannian metric on $M \cup \Omega$, also called $\tilde g$. Note that $\tilde g$ need not have nonnegative scalar curvature on $M$, but it does have nonnegative scalar curvature $\tilde R$ on $U$ that is positive on $U \setminus \Omega$. See Figure \ref{fig_M_g_tilde}.

\begin{figure}[ht]
\begin{center}
\includegraphics[scale=0.75]{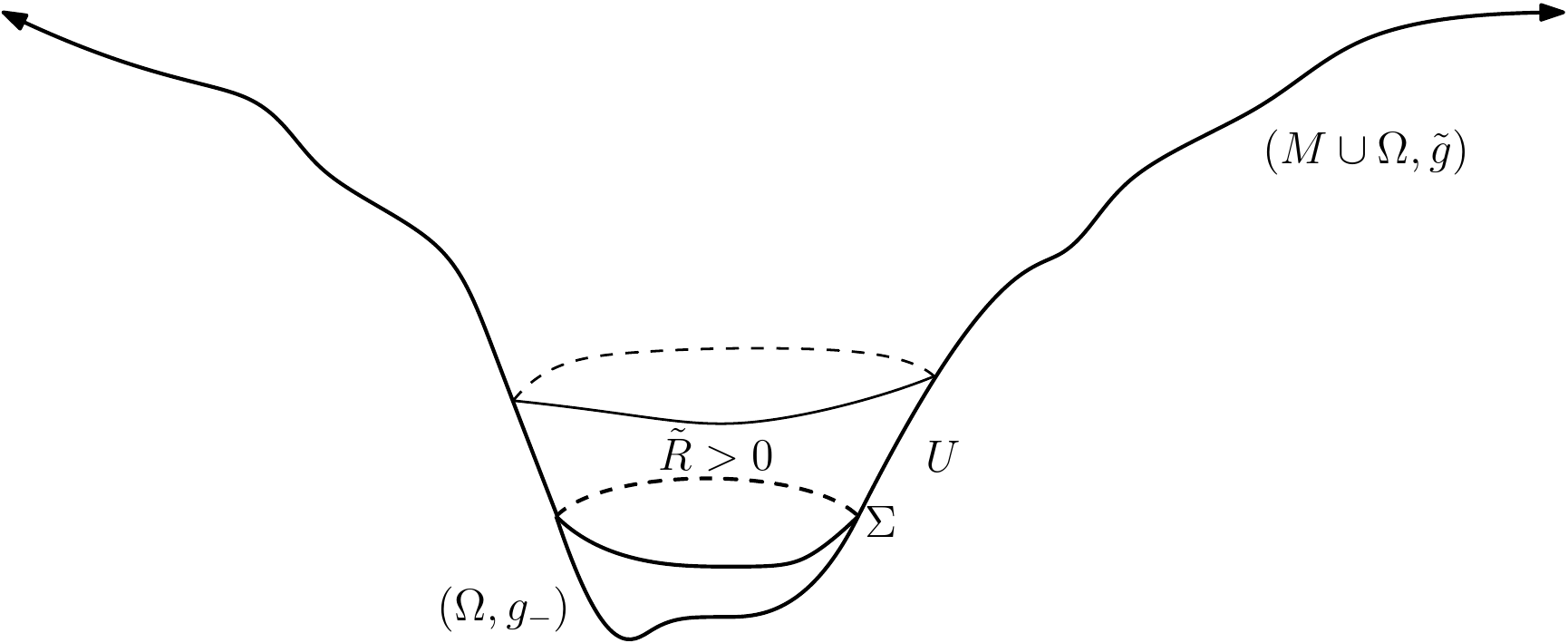}
\end{center}
\caption{\small The manifold $(M \cup \Omega, \tilde g)$ is depicted. It has mean curvature $H_-$ on $\Sigma$, and positive scalar curvature in $U \setminus \Omega$ (possibly approaching zero at $\partial \Omega$).
\label{fig_M_g_tilde} }
\end{figure}

For $t> 0$ small, let $\Sigma_t \subset M$ be the smooth distance-$t$ surface to $\Sigma$ with respect to $g_+$. 
 Let $M_t$ be the region outside of (and including) $\Sigma_t$, a smooth manifold with boundary $\Sigma_t$. Shrink $U$ if necessary so that $R_+$ is strictly positive on $\overline{U \setminus \Omega}$, so that $R_+ \geq 2\alpha $ on $\overline{U \setminus \Omega}$ for some constant $\alpha> 0$.
Shrink $t>0$ if necessary to ensure $\Sigma_{2t}$ is smooth and lies in $U$. See Figure \ref{fig_M_g_plus}.
\begin{figure}[ht]
\begin{center}
\includegraphics[scale=0.75]{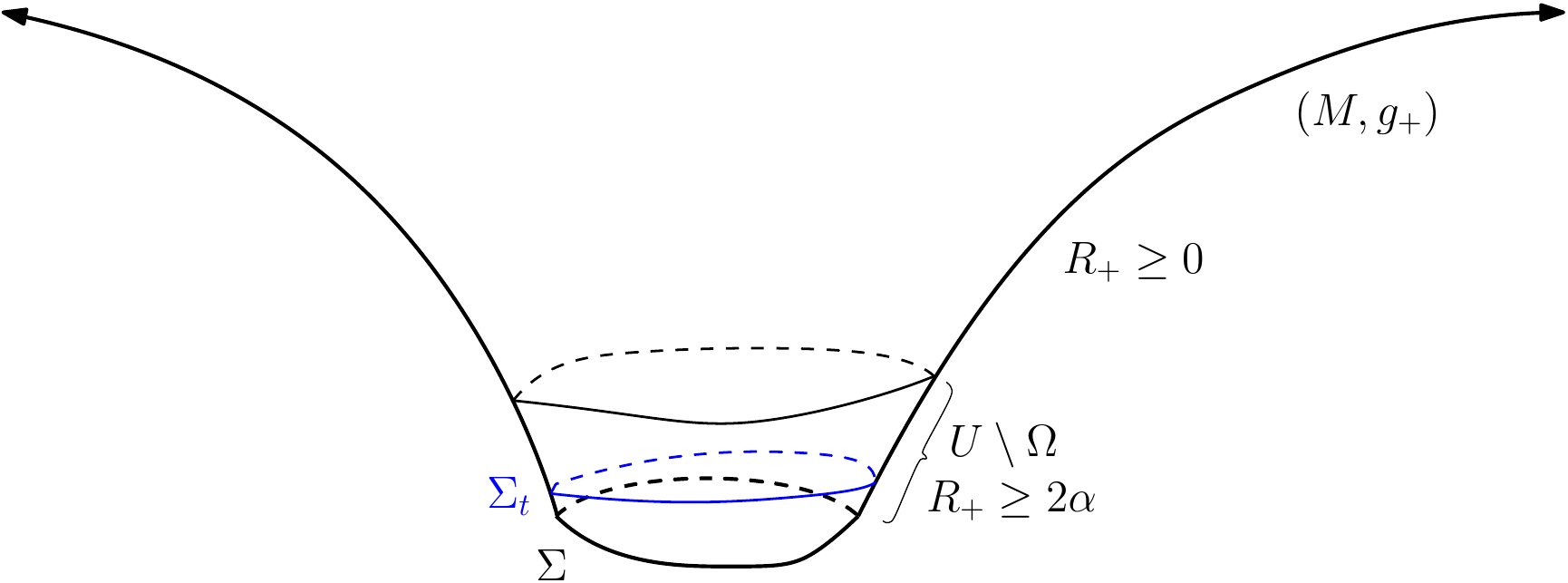}
\end{center}
\caption{\small The manifold $(M,g_+)$ is depicted. Its boundary $\Sigma$ has mean curvature $H_+$. The scalar curvature is nonnegative everywhere and bounded below by $2\alpha$ in $\overline{U \setminus \Omega}$. $\Sigma_t$ is the distance-$t$ surface from $\Sigma$.
\label{fig_M_g_plus} }
\end{figure} 

By the definition of extension, $g_+$ glues continuously to $g_-$ across $\Sigma$; in particular, $g_+ = \tilde g$ on $\Sigma_0=\Sigma$. Now, we modify $g_+$ slightly on $M_t$ so that it agrees with $\tilde g$ on the surface $\Sigma_t$ instead of on $\Sigma_0$. Let $\eta = g_+ - \tilde g$, a tensor on $M$ vanishing on $\partial M$. Let $\Phi_t: M \to M_t$ be a smooth family of diffeomorphisms, for $t$ small, such that $\Phi_0: M \to M$ is the identity and $\Phi_t$ is the identity outside of $\Sigma_{2t}$. Obviously $\Phi_t$ maps $\Sigma$ to $\Sigma_t$. On $M_t$, define the tensor
$$g_+^{(t)} = \tilde g|_{M_t} + \left(\Phi_t^{-1}\right)^* \eta.$$
Note that $g_+^{(t)}$ equals $g_+$ outside $\Sigma_{2t}$, equals $\tilde g$ on $\Sigma_t$, and converges smoothly to $g_+$ as $t \searrow 0$ (in the sense  that $\Phi_t^* g_+^{(t)} \to g_+$ smoothly on $M$ as $t \searrow 0$). 

In particular, for all $t>0$ sufficiently small, $g_+^{(t)}$ is a Riemannian metric on $M_t$ with nonnegative scalar curvature that is bounded below by $\alpha$ in $\overline{U} \cap M_t$, with ADM mass equal to that of $g_+$, with $g_+ = \tilde g$ on $\Sigma_t$.

Let $H_t^+$ be the mean curvature of $\Sigma_t$ with respect to $g_+^{(t)}$, and let $\tilde H_t$ be the mean curvature of $\Sigma_t$ with respect to $\tilde g$. As $t \searrow 0 $, we have $H_t^+ \circ \Phi_t|_{\Sigma}$ and $\tilde H_t \circ \Phi_t|_{\Sigma}$ converging uniformly to $H_+$ and $H_-, $ respectively. Thus, for $t>0$ sufficiently small, since $H_+ < H_-$, we have $H_t^+ < \tilde H_t$ on $\Sigma_t$. Now fix such a value $t$, and let $g_{++} = g_+^{(t)}$.

We now invoke a scalar curvature deformation result of Brendle--Marques--Neves \cite{BMN}, explained further in the appendix herein,  applied to the metrics $g_{++}$ and $\tilde g$ on $M_t$. This theorem applies because  $g_{++} = \tilde g$ on $\partial M_t$ and $- H_t^+ > - \tilde H_t$ (the negative here indicates mean curvature pointing out of $M_t$). Let $W \subset U \cap M_t$ be a neighborhood of $\partial M_t$ in $M_t$ on which the scalar curvatures $R_{++}$ and $\tilde R$ of $g_{++}$ and  $\tilde g$ are both bounded below by some constant $\beta>0$ (depending on $t$). For every $\delta>0$, the aforementioned result from \cite{BMN}, Theorem \ref{thm_BMN} in the appendix of this paper,
produces a  Riemannian metric $\hat g$ on $M_t$ with the following properties:
\begin{enumerate}
\item[(a)] $\hat g = g_{++}$ in $M_t \setminus W$.
\item[(b)] $\hat g = \tilde g$ in some neighborhood of $\partial M_t$.
\item[(c)]  The scalar curvature $\hat R$ of $\hat g$ satisfies $\hat R(x) \geq \min( R_{++}(x), \tilde R(x)) - \delta$ for all $x \in M_t$.
\item[(d)] $\|\hat g - g_{++}\|_{C^0(M)} < \delta$ .
\end{enumerate}

 By (b), $\hat g$ pastes smoothly to $\tilde g$ over a neighborhood of the boundary in $\partial M_t$ and thus gives a Type 1 extension of $(\Omega, g_-)$. For $\delta>0$ sufficiently small (and taking $t>0$ sufficiently small earlier in the proof), the extension is admissible, and claims (i)--(iii) of Theorem \ref{thm_main} are clear from the construction and Lemma \ref{lemma_wlg}.
 
\section{Discussion of condition $\cN$}
\label{sec_N}

In this section we attempt to give a fairly comprehensive account of the numerous versions of  condition $\cN$ in the definition of Bartnik mass that have appeared in the literature, along with a discussion of some of their advantages and disadvantages. Recall that the imposition of $\cN$ is to exclude Bartnik's examples of hiding the region $\Omega$
behind a small neck in an extension.

\begin{enumerate}
\item[(i)] Bartnik originally considered Type 1 extensions and imposed the condition $\cN=$ ``$M \cup \Omega$ has no horizons,''
where he took a horizon to be stable minimal 2-sphere \cite{Ba1}. This prohibits horizons inside $\Omega$, or even horizons that cross over $\partial \Omega$. This choice of condition $\cN$ has the advantage that if $\Omega' \subset \Omega$, then an admissible Type 1 extension of $(\Omega,g)$ satisfying $\cN$ must extend over $\Omega \setminus \Omega'$ to an admissible Type 1 extension of $\Omega'$ satisfying $\cN$. This immediately implies monotonicity of the Bartnik mass, i.e.
\begin{equation}
\label{eqn_monotonicity}
m_B^{(1)}(\Omega', g|_{\Omega'}) \leq m_B^{(1)}(\Omega, g).
\end{equation}

In this case a horizon could also have been defined, without changing the class of admissible extensions satisfying $\cN$, to be as general as an immersed compact minimal surface in $M \cup \Omega$, since the existence of such would imply the existence of a stable minimal embedded 2-sphere (upon taking the outermost minimal surface; see  \cite[Lemma 4.1]{HI} and the references therein). This uses the fact that $M \cup \Omega$ has no boundary and the metric is smooth with nonnegative scalar curvature.

One issue associated with this choice of condition $\cN$ is that it causes the value of the Bartnik mass to depend on all of $\Omega$, and not just on the \emph{Bartnik data} of $\Omega$ (i.e., the induced metric and mean curvature on $\partial \Omega$).  This is pointed out explicitly in \cite{Sz}. Generally, the expectation is that quasi-local mass ought to depend only on the Bartnik data.

Finally, we remark that if Type 2 or 3 extensions are considered for this choice of $\cN$, then care must be taken to define (zero) mean curvature for surfaces that cross through $\partial \Omega$, where the metric may be only Lipschitz.

\smallskip

\item[(ii)] G. Huisken and T. Ilmanen modified Bartnik's choice of $\cN$ by allowing $M \cup \Omega$ to have minimal boundary components, provided they are the only minimal surfaces in $M$ \cite{HI}. This requires a) a slightly different definition of allowable region $\Omega$, in which multiple  components of $\partial \Omega$ are allowed, all being minimal surfaces, except for one component $\Sigma_-$ of positive mean curvature, and b) a slightly different definition of extension $M$ in which multiple components of $\partial M$ are allowed, all being minimal surfaces, except for one component $\Sigma_+$, and $\Sigma_-$ and $\Sigma_+$ are isometrically identified. This reformulation allows minimal surfaces on the boundary of $M$ outside of $\Omega$
and also allows $\Omega$ to surround minimal surfaces; it was later adopted by Bartnik \cite{Ba4}. For such a definition, $m_B^{(1)}$ satisfies monotonicity \eqref{eqn_monotonicity}; like (i) it depends on more than just the Bartnik data of $\Omega$. Huisken and Ilmanen conjectured the conditions $\cN$ in (i) and (ii) produce the same value for the Bartnik mass \cite{HI} when $\partial \Omega = \Sigma_-$.

\smallskip

\item[(iii)] Related to (ii), a variant on Bartnik's original condition would be to require that the extension $M$ contain no horizons (cf. \cites{CCMM, Mc, Sz} for instance). This allows horizons inside $\Omega$ and crossing over $\partial \Omega$. 
Here, a horizon could be defined in a number of ways, e.g. a compact minimal surface, either immersed or embedded, possibly stable, possibly spherical in topology. Due to the presence of the boundary, however, it is not clear that all of these choices of the definition of horizon lead to identical classes of extensions: the outermost minimal surface in $M \cup \Omega$ could very well cross into $\Omega$, i.e. not remain in $M$.

On the one hand, this type of condition $\cN$ could be more desirable than that in (i), in that it is less restrictive but still excludes the problematic small minimal necks. Moreover, with this choice of $\cN$, the definitions $m_B^{(2)}$ and $m_B^{(3)}$  manifestly depend only on the Bartnik data of $\Omega$, not on the interior geometry of $\Omega$. On the other hand, there is the downside with this definition of Bartnik mass that monotonicity could conceivably fail: Suppose $(\Omega, g_-)$ is an allowable region, and $(M,g_+)$ is an admissible extension (Type 1, say) containing no compact minimal surfaces. If $ \Omega' \subset \Omega$ and $M \cup(\Omega \setminus \Omega')$ were to contain a compact minimal surface, then $M \cup (\Omega \setminus \Omega')$ would not be an allowed competitor for the Bartnik mass of $\Omega'$ satisfying this choice of condition $\cN$. 
This possible failure of the monotonicity of the Bartnik mass was also discussed by Corvino \cite{Cor2}.

\smallskip

\item[(iv)] Similar in spirit to Huisken and Ilmanen's choice of $\cN$, another possibility is to choose $\cN$ to be the property that $(M,g)$ contains no compact minimal surfaces that surround $\partial M$ (cf. \cites{AJ, CW} for instance).  (A surface \emph{surrounds} $\partial M$ if any path from infinity to $\partial M$ intersects the surface.) This definition allows extensions that contain minimal surfaces, provided they do not ``hide'' $\Omega$ from the asymptotically flat end. It is not clear that this choice of $\cN$ implies or is implied by that in (ii).

For the ``no surrounding horizons'' version of $\cN$, it does not matter whether compact minimal surfaces are taken to be immersed or embedded (or stable, or topologically spherical). For if $M$ contains an immersed compact minimal surface surrounding $\partial M$, it contains an outermost minimal surface that is a union of stable minimal embedded 2-spheres (see Lemma 4.1 of \cite{HI} and the references therein).

As with (iii), this ``surrounding'' version of $\cN$ could conceivably violate monotonicity of the Bartnik mass, but it does only depend on the Bartnik data of $\Omega$ for $m_B^{(2)}$ and $m_B^{(3)}$.
\smallskip

\item[(v)] The above possibilities for condition $\cN$ all pertain to the absence of minimal surfaces. Alternatively, Bray adopted the requirement that $\partial M$ be outward-minimizing in $(M,g_+)$ \cite{Br}, i.e. not be enclosed by a surface of less area (see also \cite{BC}), with no restriction of minimal surfaces. It is not a priori clear how such a definition for the Bartnik mass compares to the ``no horizon'' definitions. For instance, there exist allowable regions $\Omega$ in $\R^3$ that are not outward-minimizing: then $\R^3 \setminus \Omega$ is a valid competitor for (i)--(iv), but not for (v). Conversely, there exist admissible extensions $(M,g)$ for which $\partial M$ is outward-minimizing but a horizon is present. However, in Theorem \ref{thm_ineq} we show an inequality between $m_B^{(3)}$ for the outward-minimizing and the no surrounding horizons definitions.

Monotonicity of the Bartnik mass for the outward-minimizing version of $\cN$ holds (see Theorem 4.4 of \cite{BC}): if $(\Omega, g_-)$ is an allowable region and $\Omega' \subset \Omega$ is outward-minimizing in $\Omega$, then \eqref{eqn_monotonicity} holds.

Miao used the requirement that $\partial M$ be \emph{strictly} outward-minimizing
in \cite{Mi2}. Below, in the context of openness and closedness of condition $\cN$, we contrast outward-minimizing with strictly outward-minimizing. In both definitions, it is immediate the value of the Bartnik mass depends only the Bartnik data, at least for $m_B^{(2)}$ and $m_B^{(3)}$.

Unfortunately, it is not clear that the outward-minimizing and strictly outward-minimizing versions of $\cN$ lead to equivalent Bartnik masses. One would need to show that an admissible extension $(M,g_+)$ for which $\partial M$ is outward-minimizing can be perturbed, preserving nonnegative scalar curvature and whichever boundary conditions are used, to an extension for which $\partial M$ is strictly outward-minimizing. This seems to be rather delicate.

\end{enumerate}

\emph{It is not known whether any of the choices of condition $\cN$ discussed above in (i)--(v) lead to equivalent values for the Barntik mass} (for any of Type 1, 2, or 3 extensions). Trivial inequalities are possible when one condition is implied by another, e.g. no surrounding compact minimal surfaces in $M$ vs. no compact minimal surfaces in $M$. See also Theorem \ref{thm_ineq}.

We also remark that of all the versions of condition $\cN$ above, only the outward-minimizing versions seem likely to satisfy both monotonicity and dependence only on Bartnik data.

\medskip
\paragraph{\emph{Openness of condition $\cN$:}}

In some contexts, it is desirable to work with a condition $\cN$ that is open in an appropriate topology. For instance, in establishing that a Bartnik mass minimizer is static vacuum (cf. \cites{Cor1, Cor2, AJ, HMM}), it is convenient if an admissible extension satisfying $\cN$ still satisfies $\cN$ upon a small perturbation in, say, $C^{2}_{-\tau}(M)$. In \cite{Mi2}, Miao used the openness of ``$\partial M$ is strictly outward-minimizing'' in $C^{2}_{-\tau}(M)$; cf. Corollary \ref{cor_open} herein. In \cite[Lemma 2.1]{AJ}, the authors showed that ``$(M,g)$ contains no immersed compact minimal surfaces surrounding $\partial M$'' is open in $C^{2}_{-\tau}(M)$. 

It is not clear that the condition $\cN$ described in (iii) is open in $C^2_{-\tau}(M)$ (cf. the proof of \cite[Lemma 2.1]{AJ} --- without having the minimal surfaces surrounding $M$, it is not clear how to produce stable minimal surfaces and to achieve an upper bound on their areas). The strictly outward-minimizing property is open in $C^{2}_{-\tau}(M)$, but unfortunately not in $C^{0}_{-\tau}(M)$. (To see this, suppose $(M,g)$ has strictly outward-minimizing boundary. Let $\{f_i\}$ be sequence of smooth, positive functions on $M$ that equal 1 outside a fixed compact set and 1 on $\partial M$, and such that $f_i \to 1$ uniformly, but the normal derivative of $f_i$ along $\partial M$ becomes arbitrarily negative as $i \to \infty$. Then the conformal metrics $f_i^4 g$ converge to $g$ in $C^0_{-\tau}(M)$, but for $i$ large have negative boundary mean curvature (by \eqref{eqn_mc_conf}) and hence do not have outward-minimizing boundary.) The outward-minimizing property is not open even in $C^{2}_{-\tau}(M)$: if $\partial M$ is outward-minimizing but not strictly so, then a small, smooth perturbation to the metric, compactly supported in the interior of $M$, could be arranged to cause $\partial M$ to fail to be outward-minimizing.

Another reason openness of $\cN$ is desirable is in the context of smoothing a Type 2 or 3 extension to a Type 1, as in Theorem \ref{thm_main}: if the initial extension satisfies $\cN$, it would be convenient if a perturbation also satisfied $\cN$. Unfortunately perturbations in $C^{2}_{-\tau}(M)$ are not sufficient: $\partial \Omega$ will generally have different second fundamental forms with respect to $g_-$ and $g_+$, and thus one cannot expect a $C^2$ (or even $C^1$) perturbation to smooth a Type 2 or 3 extension to Type 1.
Theorem \ref{thm_main} establishes that such a smoothing can be achieved with a $C^0$ perturbation. Unfortunately, however, \emph{all} of the conditions $\cN$ presented above fail to be open with respect to $C^0$ perturbations. This makes showing the equivalence of $m_B^{(1)}$, $m_B^{(2)}$, and $m_B^{(3)}$, challenging; we only succeed in this for some conditions $\cN$; see section \ref{sec_equiv}.

\medskip
\paragraph{\emph{Closedness of condition $\cN$:}}
In other contexts, it is desirable to work with a  condition $\cN$ that is closed. For instance, when attempting to realize the infimum in the definition of the Bartnik mass \eqref{eqn_m_B}, one would like to know that a convergent sequence of admissible extensions satisfying $\cN$ also satisfies $\cN$. Closedness fails for all of the versions (i)--(iv) of Bartnik's mass,  for the $C^{2}_{-\tau}(M)$ topology (and hence for $C^0_{-\tau}(M)$). For example, in rotational symmetry, an annular region foliated by spheres of positive mean curvature could converge smoothly to a cylindrical region foliated by minimal spheres. Similarly, it is not hard to see that closedness fails for the condition that $\partial M$ be strictly outward-minimizing, for $C^2_{-\tau}(M)$. However, it is not difficult to see that the condition that $\partial M$ be outward-minimizing is closed in $C^{0}_{-\tau}(M)$.

Corvino suggests that another approach to the closedness issue is to consider limits of admissible extensions (where the limit is smooth on compact subsets of the extension), i.e. taking the closure of the class of admissible extensions \cite{Cor2}.

\medskip

In summary, there are quite a number of reasonable possible conditions $\cN$, each carrying some advantages and disadvantages. Perhaps other conditions will be proposed in the future. Among all versions of $\cN$ discussed above, none is both open and closed in a reasonable topology. In section \ref{ss1}, however, we propose a slight modification to Bray's outward-minimizing version of the Bartnik mass that captures some of the advantages of both open and closed conditions. In particular, we are able to show the equivalence, for this definition, of considering Type 1, 2, and 3 extensions in the definition of Bartnik mass.

\section{Equivalence of $m_B^{(1)}, m_B^{(2)}, m_B^{(3)}$ for outward-minimizing versions of $\cN$}
\label{sec_equiv}

In this section, we use Theorem \ref{thm_main} to establish the equivalence of  $m_B^{(1)}, m_B^{(2)}, m_B^{(3)}$ in two cases. The first (subsection \ref{ss1}) is a new variant on the outward-minimizing condition. The second (subsection \ref{ss2}) is the strictly outward-minimizing condition. 

\subsection{The $\epsilon$-outward-minimizing condition}
\label{ss1}

Here we allow extensions whose boundaries are nearly outward-minimizing (within $\epsilon$ of being so), and then let $\epsilon$ go to zero. We show the corresponding $m_B^{(i)}$ are all  equal in Theorem \ref{thm_equal_eps}.

For $\epsilon \geq 0$, let $\cN_\epsilon$ be the property ``every surface enclosing $\partial M$ has area at least $|\partial M|_g - \epsilon$'', a weakening of the outward-minimizing condition. In fact $\cN_0$ is precisely the condition ``$\partial M$ is outward-minimizing.'' For $i=1,2,3$, define:
\begin{align}
\tilde m_B^{(i)}(\Omega, g_-) &= \lim_{\epsilon \to 0^+}  \inf \left\{m_{ADM}(M,g_+) \; | \;  (M,g_+) \text{ is a Type $i$ admissible extension of } \right.\nonumber\\
&\qquad \qquad \qquad \qquad \qquad \qquad\qquad \left.(\Omega,g_-) \text{ satisfying condition } \cN_\epsilon \right\}\label{eqn_m_B_tilde}
\end{align}

Bray's modification to Bartnik's mass, dubbed the \emph{outer mass}, is precisely $m_B^{(i)}$ for $\cN=\cN_0$, i.e. the same as \eqref{eqn_m_B_tilde} with the limit and infimum interchanged (where the limit of nested sets is understood as the intersection). The  number $\tilde m_B^{(i)}$ is $\leq$ the outer mass (for Type $i$), since  $\cN_0$ implies $\cN_\epsilon$ for $\epsilon>0$. We consider it likely that equality holds, although we do not pursue this here.

Below we demonstrate that the process of considering $\cN_\epsilon$ and letting $\epsilon \to 0$ as in \eqref{eqn_m_B_tilde} has both good closedness and openness properties.

For closedness, consider the following. Let $\{(M,g_j)\}$ be a sequence of Type $i$ admissible extensions of $(\Omega, g_-)$ that is a ``minimizing sequence'' for $\tilde m_B^{(i)}(\Omega, g_-)$, where $i\in\{2,3\}$. (We fix the smooth manifold $M$ for simplicity.)  That is, the sequence $\{m_{ADM}(M, g_j)\}$ converges to $\tilde m_B^{(i)}(\Omega, g_-)$ as $j \to \infty$, and for any $\epsilon > 0$, $(M, g_j)$ satisfies $\cN_\epsilon$ for $j$ sufficiently large. Suppose $g_j$ converges to a smooth AF metric $g$ in $C^0_{-\tau}(M)$ and converges in $C^1$ to $g$ near $\partial M$ (so that the mean curvatures converge). Then $(M,g)$ is an admissible extension (of Type 2 or 3)  of $(\Omega, g_-)$ and satisfies $\cN_0$. (Nonnegative scalar curvature is preserved when taking a $C^0$ limit; see \cite[Section 1.8]{Gromov}, \cite[Theorem 1]{Bamler}.) Thus $(M,g)$ is a valid competitor for $\tilde m_B^{(i)}(\Omega, g_-)$, which justifies our claim that this 
version of the Bartnik mass has good closedness properties.
In fact, the lower semicontinuity of ADM mass results in \cites{Jau2,JL} may be used to show that the ADM mass of $g$ is at most, and hence equal to, $\tilde m_B^{(i)}(\Omega, g_-)$.

Our main justification for asserting that definition \eqref{eqn_m_B_tilde} of the Bartnik mass has good openness properties is the following result, an application of Theorem \ref{thm_main}.

\begin{thm}
\label{thm_equal_eps}
Suppose $(\Omega, g_-)$ is an extendable allowable region. Then
$$\tilde m_B^{(1)}(\Omega, g_-)=\tilde m_B^{(2)}(\Omega, g_-)=\tilde m_B^{(3)}(\Omega, g_-).$$
\end{thm}

\begin{proof}
Certainly
$$\tilde m_B^{(1)}(\Omega, g_-)\geq \tilde m_B^{(2)}(\Omega, g_-)\geq \tilde m_B^{(3)}(\Omega, g_-) \geq 0$$
is clear. We are done if $\tilde m_B^{(3)}(\Omega, g_-)=+\infty$; otherwise proceed as below.

Let $\epsilon > 0$, and take an admissible Type 3 extension $(M,g_+)$ of $(\Omega, g_-)$, satisfying $\cN_{\epsilon/2}$, whose ADM mass is less than $ \tilde m_B^{(3)}(\Omega, g_-) + \frac{\epsilon}{2}$.  By Theorem \ref{thm_main}, there exists an admissible Type 1 extension $(M,\hat g)$ of $(\Omega, g_-)$ with ADM mass less than $ \tilde m_B^{(3)}(\Omega, g_-) + \epsilon$. By choosing $\hat g$ close to $g_+$ in $C^0(M)$ as in Theorem \ref{thm_main}, we can arrange that $\hat g$ satisfies $\cN_\epsilon$. Thus,
\begin{align*}
\inf_{(M,g_+)} &\; \left\{m_{ADM}(M,g') \; | \;  (M,g') \text{ a Type $1$ admissible extension of } (\Omega,g_-) \text{ satisfying } \cN_\epsilon \right\}\\
  &\leq m_{ADM}(M, \hat g)\\
 &<  \tilde m_B^{(3)}(\Omega, g_-) + \epsilon.
\end{align*}
Taking limit $\epsilon \searrow 0$ yields the result.
\end{proof}

\begin{cor}
\label{cor_bartnik_data}
If $(\Omega, g_-)$ is an extendable allowable region, then the quantity $\tilde m_B^{(1)}(\Omega, g_-)$ depends only on the Bartnik data of $(\Omega, g_-)$.
\end{cor}

\subsection{The strictly outward-minimizing condition}
\label{ss2}
In this subsection we prove Theorem \ref{thm_equal_som}, stated in the introduction, regarding the equality of $m_B^{(1)}, m_B^{(2)}$, and $m_B^{(3)}$ for $\cN=$ ``$\partial M$ is strictly outward-minimizing.'' We first state two useful results on preserving this choice of $\cN$ under perturbations.

\begin{lemma}
\label{lemma_strict}
Let $M$ be an orientable, smooth 3-manifold, with nonempty compact boundary $\partial M$. Let $S$ be a smooth surface in the interior of $M$ that smoothly retracts onto $\partial M$. Let $O$ be a tubular neighborhood of $S$, with $\bar O$ contained in the interior of $M$. Suppose $g$ is a (continuous) asymptotically flat metric on $M$ that is smooth on $M \setminus O$, such that $\partial M$ has positive mean curvature in the direction pointing into $M$. Let $\{g_n\}$ be a sequence of smooth Riemannian metrics on $M$ such that $g_n \to g$ on $O$ in $C^0$, and $g_n \to g$ on $M \setminus O$ in $C^1_{-\tau}$.

If $\partial M$ is strictly outward-minimizing in $(M,g)$, then for all $n$ sufficiently large, $\partial M$ is strictly outward-minimizing in $(M,g_n)$.
\end{lemma}

Note that $C^1_{-\tau}$ convergence on $M \setminus O$ implies $C^1$ convergence on the bounded component of $M \setminus O$.

This lemma will be proved later in the section. The following corollary essentially states that strictly outward-minimizing is an open condition in $C^1_{-\tau}(M)$.

\begin{cor}
\label{cor_open}
Let $(M,g)$ be a smooth asymptotically flat 3-manifold, with nonempty compact boundary $\partial M$ of positive mean curvature that is strictly outward-minimizing. If $g_n$ is a sequence of asymptotically flat metrics on $M$ that converges to $g$ in $C^k_{-\tau}(M)$ for any $k \geq 1$, then for all $n$ sufficiently large, $\partial M$ is strictly outward-minimizing in $(M,g_n)$.
\end{cor}

Using these results, we now give:

\begin{proof}[Proof of Theorem  \ref{thm_equal_som}]
Let $(\Omega, g_-)$ be an extendable allowable region, and let $\epsilon > 0$. If $m_B^{(3)}(\Omega, g_-)= \infty$, then the same goes for $m_B^{(1)}(\Omega, g_-)$ and $m_B^{(2)}(\Omega, g_-)$. Thus, assume $m_B^{(3)}(\Omega, g_-)$ is finite, and let $(M,g_+)$ be a Type 3 admissible extension of $(\Omega, g_-)$, with strictly outward-minimizing boundary, such that
\begin{equation}
\label{m_b_eps}
m_{ADM}(M, g_+) < m_B^{(3)}(\Omega, g_-) + \epsilon.
\end{equation}
Without loss of generality, by Lemma \ref{lemma_wlg} we assume $(M,g_+)$ has positive scalar curvature near $\partial M$, and $H_+ < H_-$. By Corollary \ref{cor_open}, $\partial M$ may still be assumed to be strictly outward-minimizing.

We need to revisit the proof of Theorem \ref{thm_main} to establish that the strictly outward-minimizing condition can be preserved in the $C^0$ smoothing to a Type 1 extension, despite this condition not being open in $C^0_{-\tau}(M)$ (see the second paragraph of the discussion \emph{Openness of condition $\cN$} in section \ref{sec_N}). As in the proof of Theorem \ref{thm_main}, take the local extension $(U, \tilde g)$, the family of surfaces $\Sigma_t= \partial M_t$, the family of diffeomorphisms $\Phi_t$, and the Riemannian metrics $g_+^{(t)}$ on $M_t$.

We first note that for $t$ sufficiently small, $\Sigma_t$ is strictly outward-minimizing in $(M, g_+^{(t)})$. This follows from Corollary \ref{cor_open}, applied to the path of metrics $\Phi_t^* g_+^{(t)}$ on $M$ converging smoothly to $g_+$ (and equaling $g$ outside a compact set), hence converging to $g_+$ in $C^1_{-\tau}(M)$.

We claim that for $t$ sufficiently small, $\Sigma$ is strictly outward-minimizing in $M$ with respect to the Lipschitz metric obtained by gluing $\tilde g$ to $g_+^{(t)}$ along $\Sigma_t$. Call this metric $(\tilde g, g_+^{(t)})$. To show this claim observe that any competitor for the minimum enclosing area of $\partial M$ with respect to  $(\tilde g, g_+^{(t)})$ can be assumed to be enclosed by $\Sigma_t$, by the note in the previous paragraph. But for all $t$ sufficiently small, $\tilde H_t$ is positive. Then by the first variation of area formula, any such competitor can have its area strictly reduced by flowing inward towards $\Sigma$, except for $\Sigma$ itself. More precisely, the outermost minimal area enclosure of $\Sigma$ (see the proof of Lemma \ref{lemma_strict} for more details) is enclosed by $\Sigma_t$ and can contradictorily have its area decreased by such a flow, unless it equals $\Sigma$ itself. This proves the claim.

Fix such a sufficiently small value of $t>0$. We again apply the smoothing technique in \cite{BMN}, but now in conjunction with Lemma \ref{lemma_strict}. To do so, choose the surface $S$ in the lemma to be $\Sigma_t$, and fix any tubular neighborhood $O$ whose closure lies in the interior of $M$. Using Theorem \ref{thm_BMN}, there exists a sequence of smooth Riemannian metrics $\{\hat g_n\}$ on $M_t$ with nonnegative scalar curvature, with $\hat g_n = \tilde g$ on some ($n$-dependent) neighborhood of $\Sigma_t$ in $M_t$, with $\hat g_n = g_+^{(t)}$ in $M_t \setminus O$, and with $\hat g_n$ converging to $g_+^{(t)}$ in $C^0(M_t)$ as $n \to \infty$. We extend $\hat g_n$ to $M$ by smoothly gluing it to $\tilde g$. Then $\hat g_n$ converges to $(\tilde g, g_+^{(t)})$ in $C^0(M)$. By Lemma \ref{lemma_strict}, for $n$ large enough, $\Sigma$ is strictly outward minimizing in $(M, \hat g_n)$, which itself is a Type 1 admissible extension of $(\Omega, g_-)$, with ADM mass equal to that of $g_+$. Using \eqref{m_b_eps}, the proof is complete.
\end{proof}

We also state one immediate corollary of Theorem \ref{thm_equal_som}:

\begin{cor}
	\label{cor_bartnik_data2}
	If $(\Omega, g_-)$ is an extendable allowable region, then $m_B^{(1)}(\Omega, g_-)$ depends only on the Bartnik data of $(\Omega, g_-)$, if $\cN$ is chosen to be ``$\partial N$ is strictly outward-minimizing.''
\end{cor}

The essential difficulty in proving Lemma \ref{lemma_strict} is ruling out the possibility that the least-area enclosure of $\partial M$ with respect to $g_n$ reaches into $O$. The basic idea for the proof was inspired by the proof of \cite[Lemma 34]{JL}, which itself was a generalization of an argument in case 3 in the proof of \cite[Theorem 1.1]{Jau2}.

\begin{proof}[Proof of Lemma \ref{lemma_strict}]
In this proof let $\widehat M$ be a smooth manifold without boundary obtained by smoothly gluing an open 3-ball or handlebody $W$ to $\partial M$. Extend $g$ arbitrarily to a Riemannian metric on $\widehat M$ that is smooth on $\widehat M \setminus O$. By a ``surface in $M$ enclosing $\partial M$'', we mean the boundary of a bounded open set in $\widehat M$ that contains $W$. For example, $\partial M = \partial W$ is a surface in $M$ enclosing $\partial M$. We say one surface in $M$ enclosing $\partial M$ encloses another such surface if we have containment of the corresponding bounded open sets in $\widehat M$.

Let $\tilde \Sigma_n$ be the outermost minimal area enclosure of $\partial M$ in $(M,g_n)$ (which exists, has $C^{1,1}$ regularity, and is a smooth minimal surface away from $\partial M$ by standard results in geometric measure theory; see \cite[Theorem 1.3(iii)]{HI} for example). This means that $\tilde \Sigma_n$ is a surface in $M$ enclosing $\partial M$; $\tilde \Sigma_n$ has the least $g_n$-area among such; and $\tilde \Sigma_n$ encloses all other such least-area surfaces enclosing $\partial M$, if any. If $\tilde \Sigma_n = \partial M$ for all $n$ sufficiently large, then the claim holds. Thus, we pass to a subsequence (without changing the notation) so that $\tilde \Sigma_n \neq \partial M$ for all $n$.

By the $C^0$ convergence of $g_n$ to $g$, there exists a sequence of real numbers $\alpha_n \geq 1$  converging to 1 so that
 \begin{equation}
\label{alpha_n}
 \alpha_n^{-1} | \cdot |_g \leq |\cdot|_{g_n} \leq \alpha_n|\cdot|_g,
\end{equation}
where $|\cdot|$ represents area of a surface in $M$ with respect to a given metric.
 
 Let $\Sigma_t$ be the surface consisting of points in $M$ of $g$-distance $t$ from $\partial M$, smooth for $t \in [0,t_*]$, for some $t_*>0$. Let $A_t$ be the open set in $\hat M$ enclosed by $\Sigma_t$. We refer the reader to Figure \ref{fig_lemma} for illustrations of most of the objects involved in the proof.

\begin{figure}[ht]
\begin{center}
\includegraphics[scale=0.75]{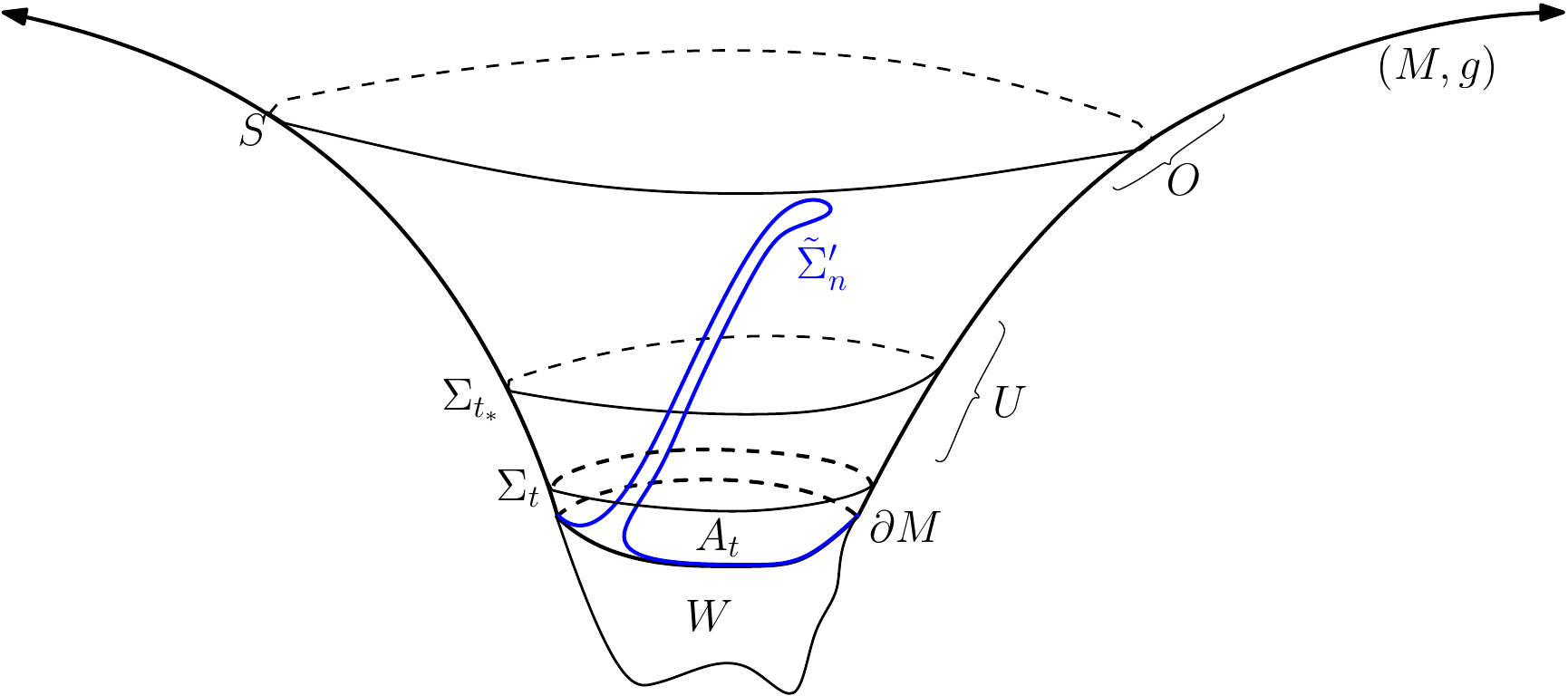}
\end{center}
\caption{\small This shows many of the sets in the proof of Lemma \ref{lemma_strict}.
\label{fig_lemma} }
\end{figure}

We proceed by establishing several claims.
\medskip
\paragraph{Claim 0:} $\displaystyle \lim_{n \to \infty} |\tilde \Sigma_n|_{g_n} =  |\partial M|_g$. Proof: On the one hand, by the definition of $\tilde \Sigma_n$,
$$  |\tilde \Sigma_n|_{g_n} \leq |\partial M|_{g_n},$$
the latter of which converges to $|\partial M|_g$. 
On the other hand,
$$ |\tilde \Sigma_n|_{g_n} \geq \alpha_n^{-1} |\tilde \Sigma_n|_g > \alpha_n^{-1} |\partial M|_g,$$
since $\partial M$ is strictly outward-minimizing in $(M,g)$.
Together, these imply Claim 0.

\medskip
\paragraph{Claim 1:}  $\displaystyle \liminf_{n \to \infty} d_g(\tilde \Sigma_n, \partial M) = 0$ (where $d_g(\cdot, \cdot)$ is the minimum distance between the respective surfaces with respect to $g$). Proof: If not, pass to a subsequence for which
\begin{equation}
\label{d_c}
d_g(\tilde \Sigma_n, \partial M)  \geq c
\end{equation}
for some constant $c>0$ independent of $n$. There exists some surface $\sigma$ in $M$ enclosing $\partial M$ that i) does not intersect $\partial M$, ii) is contained in $A_c$, and iii) is strictly outward-minimizing with respect to $g$. (One way to construct $\sigma$ is to flow $\partial M$ by inverse mean curvature flow for a short time, and use the ``smooth start lemma'' (Lemma 2.4) of Huisken--Ilmanen \cite{HI} to assure that the flowed surface remains strictly outward-minimizing.) By \eqref{d_c}, $\tilde \Sigma_n$ encloses $\sigma$ for all $n$.
We now obtain a string of inequalities:
$$|\partial M|_g < |\sigma|_g \leq |\tilde \Sigma_n|_g \leq \alpha_n |\tilde \Sigma_n|_{g_n} \leq \alpha_n |\partial M|_{g_n}.$$
Taking $\displaystyle\lim_{n \to \infty}$ leads to a contradiction, proving Claim 1.

\medskip
Using Claim 1,  we now pass to a subsequence for which $d_g(\tilde \Sigma_n, \partial M)$ converges to $0$ as $n \to \infty$. In particular, $\tilde \Sigma_n$ has a connected component $\tilde \Sigma_n'$ for each $n$, such that the sequence $d_g(\tilde \Sigma_n', \partial M)$ converges to 0. 

\medskip
\paragraph{Claim 2:} Possibly shrinking $t_*>0$, $\tilde \Sigma_n'$ intersects $\Sigma_{t_*}$ for all $n$ sufficiently large. Proof: choose $t_*$ sufficiently small so that $\Sigma_{t_*}$ lies in the bounded component of $M \setminus \overline{O}$, and the mean curvature of $\Sigma_t$ with respect to $g$ is strictly positive for all $t \in [0,t_*]$. Since $g_n \to g$ in $C^1$ on this component, we may choose $N$ sufficiently large so that the mean curvature of $\Sigma_t$ with respect to $g_n$ is strictly positive for all $t \in [0,t_*]$ and all $n \geq N$.

By Claim 1, we may increase $N$ if necessary so that $d_g(\tilde \Sigma_n', \partial M) < t_*$ for all $n \geq N$. Now, for $n \geq N$, if $\tilde \Sigma_n'$ did not intersect $\Sigma_{t_*}$, then by connectedness, $\tilde \Sigma_n'$ would lie inside $A_{t_*}$. Note $\tilde \Sigma_n' \neq \partial M$, or else $\tilde \Sigma_n = \partial M$ as well, contrary to the hypothesis. Thus, $\tilde \Sigma_n'$ has a point in the interior of $M$ where locally it is a minimal surface with respect to $g_n$ and is tangent to some $\Sigma_t$, for some $t \in (0,t_*)$. This is a contradiction to the standard comparison principle for mean curvature. (Specifically, we mean that if two smooth surfaces are tangent at a point, and one locally lies to one side of the other, then there is an inequality between their mean curvatures at the point of tangency. This can be seen by writing the surfaces locally as a graph and looking using the mean curvature formula for graphs.)

\medskip
Using Claim 2,  we now fix such a $t_*$ and truncate finitely many terms of the sequence so that now $\tilde \Sigma_n'$ intersects $\Sigma_{t_*}$ for all $n$.

\medskip
\paragraph{Claim 3:} Let $U$ be a tubular neighborhood about $\Sigma_{t_*}$ of sufficiently small $g$-radius $r_0< t_*$ so that $\bar U \subset \interior(M) \setminus \bar O$, as shown in Figure \ref{fig_lemma}. Then (we claim) there exists some $b>0$, independent of $n$, so that
$$|\tilde \Sigma_n \cap U|_{g_n} \geq b$$
for all $n$. Proof: By hypothesis, $S$ and $\partial M$ bound a region $R$ that can be smoothly embedded in $\R^3$; thus we may consider the Euclidean metric $\delta$ on $R$ (which contains $\bar U$). By \eqref{alpha_n}, and since $g$ and $\delta$ are uniformly equivalent on $R$, there exists some $C> 1$ so that
\begin{equation}
\label{eqn_C}
C^{-1} | \cdot |_\delta \leq |\cdot|_{g_n} \leq C|\cdot|_\delta
\end{equation}
on $R$, for all $n$.
Since $\tilde \Sigma_n'$ is an area-minimizer in $\interior(R)$ with respect to $g_n$, then we have $\tilde \Sigma_n'$ is a $\gamma$-almost-minimizer\footnote{Recall that for a real number $\gamma\geq 1$, an integral current $T$ in $\R^n$ is \emph{$\gamma$-almost area-minimizing} if, for any ball $B$ with $B\cap\spt \partial T=\emptyset$ and any integral current $T'$ with $\partial T'=\partial(T\llcorner B)$, we have $|T\llcorner B|\leq\gamma |T'|$. Here, $T \llcorner B$ is the restriction of $T$ to $B$. Since the integral currents we deal with are submanifolds, we will abuse notation slightly and use the intersection in place of restriction.} of area in $R$ with respect to $\delta$, for the value $\gamma=C^2$. 

By Claim 2, $\tilde \Sigma_n'$ intersects $\Sigma_{t_*}$ for each $n$, say at point $p_n$. Using the monotonicity formula for $\gamma$-almost-minimizing currents (see, for example, \cite[Lemma 5.1]{BL}), we have for some $r_1>0$ (independent of $n$),
$$|\tilde \Sigma_n' \cap B(p_n, r_1)|_{\delta} \geq 4\pi\gamma^{-1}r_1^2,$$
where the ball $B(p_n,r_1)$ is taken with respect to $\delta$, and $B(p_n,r_1) \subset U$.
Using this and \eqref{eqn_C},
$$|\tilde \Sigma_n \cap U|_{g_n} \geq |\tilde \Sigma_n' \cap B(p_n, r_1)|_{g_n} \geq 4\pi\gamma^{-1} C^{-1}r_1^2.$$
This implies Claim 3.

\medskip
\paragraph{Claim 4:} If $\tilde \Omega_n$ is the bounded region in $\widehat M$ that $\tilde \Sigma_n$ bounds, then $\displaystyle \liminf_{n \to \infty} \vol_g(\tilde \Omega_n \setminus W)=0$. Proof: Since $g$ is asymptotically flat and $g_n \to g$ in $C^1_{-\tau}$ outside a compact set, we have that all $\tilde \Omega_n$ remain inside a fixed compact set in $\widehat M$. By the $C^0$ convergence of $g_n$ to $g$ and the fact $|\tilde \Sigma_n|_{g_n} \leq |\partial M|_{g_n}$ (the latter of which converges to $|\partial M|_g$), the $g$-areas of $\tilde \Sigma_n$ and the $g$-volumes of $\tilde \Omega_n$ are uniformly bounded above. Since $M$ is orientable, we can view each $\tilde \Omega_n$ in a natural way as an integral 3-current of multiplicity 1 on $\widehat M$. 

By the Federer--Fleming compactness theorem for integral currents \cite[4.2.17]{Fed}, a subsequence of $\tilde \Omega_n$ converges in both the flat and weak sense to some integral 3-current $\tilde \Omega$ of multiplicity 1 (which can also be viewed as a bounded open set in $\widehat M$ that contains $W$). In particular, i) the boundaries $\partial \tilde \Omega_n = \tilde \Sigma_n$ converge in the weak sense to $\partial \tilde \Omega$ as integral 2-currents, and ii) the $g$-volume of the symmetric difference $\tilde \Omega_n \,\triangle\, \tilde \Omega$ converges to zero. By i), the lower semicontinuity of areas for weak convergence, and Claim 0, we have
$$|\partial M|_g = \lim_{n \to \infty} |\tilde \Sigma_n|_{g_n} \geq |\partial \tilde \Omega|_g.$$
But since $\partial M$ is strictly outward-minimizing in $(M,g)$, we must have $\partial \tilde \Omega = \partial M$, i.e., $\tilde \Omega = W$. Then by ii), the $g$-volume of $\tilde \Omega_n \setminus W$ converges to zero.

\medskip
Using Claim 4,  we now pass to a subsequence for which $\vol_g(\tilde \Omega_n \setminus W)$ converges to 0 as $n \to \infty$.

\medskip
\paragraph{Claim 5:} For almost all $t \in (0, t_*]$, $\displaystyle\lim_{n \to \infty} |\Sigma_t \cap \tilde \Omega_n|_{g_n}=0$.  Proof: 
First,  $\displaystyle\lim_{n \to \infty} |\Sigma_t \cap \tilde \Omega_n|_{g}=0$ for almost all $t \in (0, t_*]$ follows from Claim 4 and the slicing theorem for normal currents \cite[4.2.1]{Fed}. With this, Claim 5 follows from \eqref{alpha_n}.

\medskip
\paragraph{Completion of the proof of the lemma:} We arrive at a contradiction as follows. Fix a value of $t \in (0,t_* - r_0)$ that satisfies the conclusion of Claim 5. Note that $\Sigma_t$ is contained in the region between $\bar U$ and $\partial M$. 

We claim that $\partial (\tilde \Omega_n \cap A_t)$ has strictly less $g_n$-area than $\tilde \Sigma_n$ (for large $n$), which contradicts the latter being an area-minimizer among boundaries of regions that contain $\partial M$.

To see this, we have on the one hand
\begin{equation}
\label{one}
|\partial (\tilde \Omega_n \cap A_t)|_{g_n} = |\tilde \Sigma_n \cap A_t|_{g_n} + |\Sigma_t \cap \tilde \Omega_n|_{g_n}.
\end{equation}
On the other hand,
\begin{align}
|\tilde \Sigma_n|_{g_n} &= |\tilde \Sigma_n \cap A_t|_{g_n} + |\tilde \Sigma_n \setminus A_t|_{g_n} \nonumber\\
&\geq  |\tilde \Sigma_n \cap A_t|_{g_n} + |\tilde \Sigma_n \cap U|_{g_n}\nonumber\\
&\geq  |\tilde \Sigma_n \cap A_t|_{g_n} +b \label{two},
\end{align}
by Claim 3. 
Comparing statements \eqref{one} and \eqref{two} and using Claim 5, we see that $\partial (\tilde \Omega_n \cap A_t)$ has strictly less $g_n$-area than $\tilde \Sigma_n$ for all $n$ sufficiently large, a contradiction. This completes the proof of the lemma.
\end{proof}

\section{An inequality between the no-horizons and the outward-minimizing Bartnik masses}
\label{sec_ineq}

As discussed in section \ref{sec_N}, point (v), there is no a priori comparison between the ``no-horizons'' and the outward-minimizing definitions of the Bartnik mass. To reiterate, on the one hand, it is possible for $M$ to contain no compact minimal surfaces and yet for
 $\partial M$ to fail to be outward-minimizing (a point which is occasionally overlooked). On the other hand, it is possible for $\partial M$ to be outward-minimizing and yet contain a compact minimal surface surrounding $\partial M$.
 
In this section we prove an inequality between the ``no surrounding horizons'' and the ``outward-minimizing'' versions of the Bartnik mass. The quantity $\lambda_1(-\Delta +K)$ below will be  the lowest eigenvalue of $-\Delta + K$ ($K$ being the Gauss curvature)  on $\partial \Omega$.

\begin{thm}
\label{thm_ineq} 
Let $(\Omega,g_-)$ be an allowable region for which $\lambda_1(-\Delta +K) > 0$ on $\partial \Omega$, with $\partial \Omega$ topologically a 2-sphere. Then the value of $m_B^{(3)}(\Omega, g_-)$ defined for $\cN$ = ``$\partial M$ is outward-minimizing in $M$'' (or strictly outward-minimizing) is greater than or equal to its value defined for $\cN$ = ``$M$ contains no compact minimal surfaces surrounding $\partial M$.''
\end{thm}

\begin{proof}
Call these $m_B^{(3)}(\Omega)$ values $m_o$ and $m_h$, respectively, and let $\mu = \sqrt{\frac{|\partial \Omega|_{g_-}}{16\pi}}$. It is sufficient to prove the inequality for the non-strict outward-minimizing condition.

We consider three cases, according to how the values of $m_o$ and $\mu$ compare. Let $\epsilon >0$.

\smallskip

First, suppose $m_o <\mu$. Then by the definition of $m_B^{(3)}$, there exists a Type 3 admissible extension of $(\Omega, g_-)$, say $(M,g_+)$,  for which $\partial M$ is outward-minimizing, such that
\begin{align}
m_{ADM}(M,g_+) &< \mu, \text{ and} \label{eqn_adm_ineq}\\
m_{ADM}(M,g_+) &< m_o + \epsilon. \nonumber
\end{align}
Suppose that $(M,g_+)$ contains a compact minimal surface surrounding $\partial M$. Then $(M,g_+)$ contains an outermost minimal surface $S$ surrounding $\partial M$, which is outward-minimizing (again, see \cite[Lemma 4.1]{HI} and the references therein). Then by the Riemannian Penrose inequality (particularly Bray's version that allows for $S$ to be disconnected \cite{Br}), we have
$$m_{ADM}(M,g_+) \geq \sqrt{\frac{|S|_{g_+}}{16\pi}} \geq  \sqrt{\frac{|\partial M|_{g_+}}{16\pi}}=\mu,$$
where the second inequality holds because $\partial M$ is outward-minimizing. This contradicts \eqref{eqn_adm_ineq}.

Thus $(M,g_+)$ has no compact minimal surfaces surrounding $\partial M$, so it is a valid competitor for the  definition of Bartnik mass corresponding to $m_h$. 
In particular
$$m_h \leq m_{ADM}(M,g_+) < m_o + \epsilon.$$ 
Letting $\epsilon \searrow 0$, the proof in the first case is complete.

\smallskip

Second, suppose that $m_o = \mu$. In \cite[Theorem 2.1]{MS}, Mantoulidis and Schoen construct an admissible extension $(M, g_+)$ of $(\Omega, g_-)$ for which $\partial M$ has zero mean curvature and is strictly outward-minimizing, such that  $M \setminus \partial M$ has a foliation by surfaces of positive mean curvature, and for which the ADM mass of $(M,g_+)$ is at most $ \mu + \epsilon$. (This uses the hypothesis $\lambda_1(-\Delta +K) > 0$ and the 2-sphere topology).  

We will perturb $g_+$ slightly so that $\partial M$ is not minimal. To do so, solve the linear elliptic problem, for a constant $a>1$:
$$\begin{cases}
\Delta u = 0 & \text{ on } (M,g_+)\\
u \to a & \text{ at infinity}\\
u =1 & \text{ on } \partial M.
\end{cases}$$
There exists a unique solution that is smooth and positive. By similar reasoning as in the proof of Lemma \ref{lemma_wlg}, the conformal metric $g'=u^4g_+$ is asymptotically flat with nonnegative scalar curvature. Moreover, $g'$ induces the same metric on $\partial M$ as $g_+$ and endows $\partial M$ with positive mean curvature $H'$. For $a>1$ sufficiently close to 1, the ADM mass of $g'$ is within $\epsilon$ that of $g$ (as follows from \eqref{eqn_adm_change}), $H'$ is pointwise less than $H_-$, and $(M,g')$ has a foliation by surfaces of positive mean curvature.  Thus, $(M,g')$ is a Type 3 admissible extension  of $(\Omega, g_-)$ that contains no compact minimal surfaces at all (by the comparison principle for mean curvature).  This shows 
$$m_h \leq m_{ADM}(M,g') < m_{ADM}(M, g_+) + \epsilon \leq \mu + 2\epsilon = m_o + 2\epsilon.$$
Since $\epsilon>0$ is arbitrary, this completes the proof in the second case.

For the third case, $m_o > \mu$, shrink $\epsilon$ if necessary so that $0 < \epsilon < m_o - \mu$. Taking the admissible extension $(M,g_+)$ as in \cite[Theorem 2.1]{MS} of ADM mass at most $\mu +\epsilon$, we obtain a Type 3 admissible extension of $(\Omega, g_-)$ for which the boundary is outward-minimizing. Then
$$m_o \leq m_{ADM}(M,g_+) \leq  \mu + \epsilon,$$
a contradiction.
\end{proof}

We point out that the third case of the proof above also gives the following coarse upper bound for the Bartnik mass.

\begin{prop}
\label{prop_bmass_estimate}
Let $\cN$ be one of the following: ``$\partial M$ is strictly outward-minimizing in $M$,'' ``$\partial M$ is outward-minimizing in $M$,'' ``$M$ contains no compact minimal surfaces,'' or ``$M$ contains no compact minimal surfaces surrounding $\partial M$.'' Then for any allowable region $(\Omega, g_-)$, with $\lambda_1(-\Delta +K)>0$ and $\partial \Omega$ topologically spherical,
\begin{equation}
\label{eqn_bmass_estimate}
m_B^{(3)}(\Omega, g_-) \leq \sqrt{\frac{|\partial \Omega|_{g_-}}{16\pi}},
\end{equation}
where $m_B^{(3)}$ is taken with any of the above choices of $\cN$.
\end{prop}

\begin{remark}
We conjecture that strict inequality holds in \eqref{eqn_bmass_estimate} (recalling that $H_->0$ is assumed in the definition of allowable region), but do not pursue this here. When $H_-$ is sufficiently large, this may follow from Bartnik mass estimates in \cites{LS,CCMM}.
\end{remark}

\section*{Appendix: scalar curvature deformation theorem of Brendle--Marques--Neves}
Here we recall a theorem of Brendle, Marques, and Neves that was used in their counterexamples to the Min-Oo conjecture:

\begin{thm}(Theorem 5 of \cite{BMN}.)
\label{thm_BMN}
Let $M$ be a smooth manifold with compact boundary $\partial M$, and let $g$ and $\tilde g$ be two smooth Riemannian metrics on $M$ such that $g -\tilde g\equiv 0$ on $\partial M$. Assume the mean curvatures of $\partial M$ (in the direction pointing out of $M$) satisfy $H > \tilde H$. Then given any $\epsilon>0$ and any neighborhood $U$ of $\partial M$, there exists a smooth Riemannian metric $\hat g$ on $M$ with the following properties:
\begin{itemize}
\item $\hat R(x) \geq \min\{R(x), \tilde R(x)\} - \epsilon$ for each $x \in M$ (where $R, \tilde R,$ and $\hat R$ are the scalar curvatures of $g, \tilde g,$ and $\hat g$, respectively).
\item $\hat g = g$ on $M \setminus U$.
\item  $\hat g = \tilde g$ on some neighborhood of $\partial M$.
\end{itemize}
\end{thm}
Although this result is stated in \cite{BMN} for compact manifolds with boundary, the construction is localized near the boundary, so that only a compact boundary is required.

One fact we need, beyond what is explicitly stated in Theorem \ref{thm_BMN}, is that the deformation can be made arbitrarily small in $C^0$ norm relative to $g$. To see this, we will require a few details of the construction. Let $\rho \geq 0$ be a smooth function on $M$ with $\rho^{-1}(0) =\partial M$, $|\nabla \rho|_g = 1$ on $\partial M$, and $\rho \geq 1$ outside a compact set containing $\partial M$. There exists a smooth, symmetric, covariant 2-tensor $T$ on $M$, vanishing on $M \setminus U$, so that on some neighborhood of $\partial M$, $\tilde g = g + \rho T$. There are also smooth cut-off functions $\chi, \beta$ whose details are not important here, except that they are bounded. The metric $\hat g$ will be chosen to be $\hat g_\lambda$, defined below, for a sufficiently large value $\lambda$:
\begin{equation*}
\hat g_\lambda = \begin{cases}
g + \lambda^{-1} \chi(\lambda \rho) T, & \text{ for } \rho \geq e^{-\lambda^2},\\
\tilde g  - \lambda \rho^2 \beta(\lambda^{-2} \log \rho) T, & \text{ for } \rho < e^{-\lambda^2},
\end{cases}
\end{equation*}

We claim that $\|\hat g_\lambda - g\|_{C^0(M)} \to 0$ as $\lambda \to \infty$. For the region $\{\rho \geq e^{-\lambda^2}\}$, ${\|\hat g_\lambda-g\|_{C^0(M)}}$ is $O(\lambda^{-1})$, since $\chi$ and $T$ are bounded independently of $\lambda$. On  ${\{\rho < e^{-\lambda^2}\}}$, $\|\hat g_\lambda -\tilde g\|_{C^0(M)}$ is $O(\lambda e^{-2\lambda^2})$. However, $\tilde g$ and $g$ agree on $\partial M$, and $\{\rho < e^{-\lambda^2}\}$ becomes arbitrarily close to $\partial M$, so we have convergence  of $\hat g_\lambda$ to $g$ in $C^0(M)$ as $\lambda \to \infty$.

\end{document}